\documentclass[11pt,a4paper]{amsart}
\usepackage{amscd,amssymb,amsmath,latexsym}

\usepackage{amsfonts}
\usepackage{graphicx}%
\usepackage[active]{srcltx}
\newtheorem{definition}{Definition}[section]
\newtheorem{notation}[definition]{Notation}
\newtheorem{rema}[definition]{Remark}
\newtheorem{exa}[definition]{Example}
\newtheorem{obs}[definition]{Observation}

\newtheorem{lemma}[definition]{Lemma}
\newtheorem{proposition}[definition]{Proposition}
\newtheorem{theorem}[definition]{Theorem}

\def\ov#1{{\overline{#1}}}

\newenvironment{remark}%
{\begin{rema}\rm}%
{\end{rema} }

\newenvironment{example}%
{\begin{exa}\rm}%
{\end{exa} }

{\begin{obs}\rm}%
{\end{obs} }

\newenvironment{ac}{\noindent{\bf Acknowledgements:}}{}

\newcommand{\Sres}{{\operatorname{Sres}}}

\newcommand{\Id}{{\operatorname{Id}}}

\newcommand{\bfy}{{\boldsymbol{y}}}

\def\U{{\mathcal U}}

\def\N{{\mathbb N}}

\def\Z{{\mathbb Z}}

\begin{document}

\title[Subresultants and Cauchy interpolation]{Subresultants, Sylvester sums and the rational interpolation problem}

\author{Carlos D' Andrea}
\address{Universitat de Barcelona, Departament d'{\`A}lgebra i Geometria.
Gran Via 585, 08007 Barcelona, Spain.}
\email{cdandrea@ub.edu}
\urladdr{http://atlas.mat.ub.es/personals/dandrea}

\author{Teresa Krick}
\address{Departamento de Matem\'atica, Facultad de
Ciencias Exactas y Naturales, Universidad de Buenos Aires and IMAS,
CONICET, Argentina} \email{krick@dm.uba.ar}
\urladdr{http://mate.dm.uba.ar/\~\,krick}

\author{Agnes Szanto}
\address{Department of Mathematics, North Carolina State
University, Raleigh, NC 27695 USA}
\email{aszanto@ncsu.edu}
\urladdr{www4.ncsu.edu/\~\,aszanto }

\begin{abstract}
We present a solution for the classical univariate rational
interpolation problem by means of (univariate) subresultants. In the
case of Cauchy interpolation (interpolation without multiplicities),
we give explicit formulas for the solution in terms of symmetric
functions of the input data,  generalizing the well-known formulas
for Lagrange interpolation. In the case of the osculatory rational
interpolation  (interpolation with multiplicities),  we give
determinantal expressions in terms of the input data, making
explicit some matrix formulations that can  independently be derived
from  previous results by Beckermann and Labahn.
\end{abstract}

\date{\today}
\thanks{Carlos D'Andrea is partially supported
by  the Research Project MTM2007--67493, Teresa Krick is partially
suported by ANPCyT PICT 2010, CONICET PIP 2010-2012 and UBACyT
2011-2014, and Agnes Szanto was partially supported by NSF grants
CCR-0347506 and CCF-1217557.}\maketitle

\maketitle

\section{Introduction}
The {\em Cauchy interpolation problem} or {rational interpolation
problem}, considered already in \cite{cau41, ros45,pred53}, is the
following:   \begin{quote} Let $K$ be a field, $a,b\in \Z_{\geq0},$
and set $\ell=a+b$. Given  a set $\{x_0,\dots,x_\ell\}$ of $\ell +
1$ distinct points in $K$,
 and $y_0,\dots,y_\ell \in K$, determine --if possible-- polynomials  $A
,\,B \in K[x]$ such that $\deg(A)\leq a,\,\deg(B)\leq b$ and
\begin{equation}\label{uu}
\frac{A}{B}\,(x_i)=y_i,\ 0\le i\le \ell.
\end{equation}
\end{quote}

This might be considered as a generalization of the classical
Lagrange interpolation problem for polynomials, where $b=0$ and
$a=\ell$. In contrast with that case, there is not always a solution
to this problem, since for instance by setting $y_0=\dots=y_a=0$,
the numerator $A$ is forced to be identically zero, and therefore
the remaining $y_{a+k}$,  $1\leq k\leq\ell-a$, have to be zero as
well. However, when there is a solution,  then the rational function
$A/B$ is unique as shown below.
\\

The obvious generalization of the Cauchy interpolation problem receives the name
 {\em osculatory rational interpolation problem} or rational Hermite interpolation problem:

\begin{quote}  Let $K$ be a field, $a,b\in \Z_{\geq0},$ and set
$\ell=a+b$. Given  a set $\{x_0,\dots,x_k\}$ of $k + 1$ distinct
points in $K$, $a_0,\ldots, a_k \in \Z_{\geq0}$ such that $a_0+\dots
+ a_k=\ell+1,$ and $y_{i,j}\in K$, $0\leq i\leq k$ , $0\leq j <
a_i$, determine --if possible-- polynomials  $A ,\,B \in K[x]$ such
that $\deg(A)\leq a,\,\deg(B)\leq b$ and
\begin{equation}\label{uv}
\left(\frac{A}{B}\right)^{(j)}(x_i)=j!\,y_{i,j},\ 0\le i\le k,\ 0\le
j<a_i.
\end{equation}
\end{quote}

 This problem has also been extensively studied from both an
algorithmic and theoretical point of view,
 see for instance  \cite{sal62,kah69,wuy75,BL00,TF00} and the references therein.
  A unified framework,  which relates the rational interpolation problem with the Euclidean algorithm, is presented in \cite{Ant88}, 
  and also  in the book \cite[Section 5.7]{vzGJ03}, where it is called called {\em rational function
  reconstruction}. 
In Theorem \ref{t3} below, we translate these results to
  the subresultants context, which enables us to obtain  some explicit
   expressions in terms of the input data for both problems.

For the Cauchy interpolation problem, there exists an explicit
closed formula in terms of the input data that can be derived from
the results on symmetric operators in a suitable ring of polynomials
presented  in \cite{las}, as shown in \cite{Las}.  Theorem
\ref{mainroots} recovers this expression from the relationship
between  subresultants and  the Sylvester sums introduced by
Sylvester in \cite{sylv}, see also \cite{LP03,DHKS07, DHKS09,
RS11,KS12}.

We also present in Theorem \ref{determinantal} an explicit
determinantal expression for the solution of the osculatory rational
 interpolation problem  in terms of the input data, giving it as  a
quotient  of determinants of generalized Vandermonde-type (and
Wronskian-type) matrices. This generalizes straight-forwardly the
corresponding known determinantal expression for the classical
Hermite interpolation problem, setting another unified framework for
all these interpolation problems. As mentioned in Remark \ref{LB}
below, this determinantal expression can actually also be derived
following the work of  Beckermann and Labahn in \cite{BL00}, as we
concluded from a recent useful discussion with George Labahn.

Since no  closed formula for subresultants in terms of roots with multiplicities  is known yet --except for very few exceptions,
see \cite{DKS13}-- a generalization
 of Theorem \ref{mainroots} to the osculatory rational interpolation problem is  still missing, and some more work
 on the subject must be done in order to shed light to the  problem.

\section{Subresultants and the rational interpolation problem}

Let us start by showing that a solution $A/B$  for the rational
interpolation problem, when it exists, is unique.

\begin{proposition}\label{canonical} If the osculatory rational interpolation problem \eqref{uv} has a solution, then  there exists  a unique pair $(A,B)$ with  $\gcd(A,B)=1 $ and $A$
monic
such that $A/B$ is a
 solution.

\end{proposition}

\begin{proof} If there is a solution, then, cleaning common factors and dividing by the leading coefficient of $A$, there is a  solution
 satisfying the same degree bounds with $\gcd(A,B)=1 $ and $A$ monic. Assume
$A_1/B_1$ and $A_2/B_2$ are both solutions of the same type. Then,
$(A_1/B_1)^{(j)}(x_i)=(A_2/B_2)^{(j)}(x_i)$ implies
$$\left(\frac{A_1B_2-A_2B_1}{B_1B_2}\right)^{(j)}(x_i)=0\mbox{  for } \ 0\le i\le k,\, 0\le
j<a_i,$$
 which inductively implies that
$(A_1B_2-A_2B_1)^{(j)}(x_i)=0$ for the $\ell+1$ conditions. But
$A_1B_2-A_2B_1$ is a polynomial of degree at most $\ell$, and
therefore $A_1B_2=A_2B_1$. Therefore, $A_1=cA_2$ and $B_1=cB_2$ with $c\in K\setminus\{0\}$. Both $A_1$ and $A_2$ are monic, so $c=1$ and the claim follows.
\end{proof}

Our results are
consequences of  interpreting  the rational interpolation
problem in terms  of conditions of subresultants of the following two
polynomials:
\begin{itemize}
 \item $f:=\prod_{j=0}^k(x-x_j)^{a_j}, $ which we write $f=\sum_{i=0}^{\ell+1} f_i
 x^i$. Note that  $f_{\ell+1}=1$.
 \item $g=\sum_{i=0}^{\ell} g_i x^i\in K[x]$, the {Hermite interpolation polynomial}
associated to the input data $(\overline X, Y)$ (where we assume
$g_i=0$ for $\deg(g)<i\le \ell$ in case $\deg(g)<\ell$).
\end{itemize}

 \smallskip We can assume in what follows  that at least one of the  $y_{i,j}$ is non-zero, as otherwise the solution of the rational interpolation problem is the $0$ function.

\smallskip For $d\le \ell$, consider the $d$-th subresultant polynomial
{\small $\Sres_d(f,g) $ of $f$ and $g$, defined as
\begin{equation}\label{defsub}
\Sres _d(f,g)
:=\det%
\begin{array}{|cccccc|c}
\multicolumn{6}{c}{\scriptstyle{2\ell+1-2d}}\\
\cline{1-6}
f_{\ell+1} & \cdots & & \cdots & f_{d+1-\left(\ell-d-1\right)}& x^{\ell-d-1}f(x)&\\
& \ddots & && \vdots  & \vdots &\scriptstyle{\ell-d}\\
&  &f_{\ell+1}&\cdots &f_{d+1}& x^0f(x)& \\
\cline{1-6}
g_{\ell} &\cdots & & \cdots & g_{d+1-(\ell-d)}&x^{\ell-d}g(x)&\\
&\ddots &&&\vdots & \vdots &\scriptstyle{\ell+1-d}\\
&& g_{\ell} &\cdots & g_{d+1} & x^0g(x)&\\
\cline{1-6} \multicolumn{2}{c}{}
\end{array}\ .
\end{equation}}
Note that the previous definition makes sense even if
$\deg(g)=m<\ell$, and agrees for $d\le m$ with the usual definition
of subresultant of $f$ and $g$ given by the matrix of the right size
$\ell+1+m-2d$, since $f$ is monic. For $m<d<\ell$ we have, according
to the definition above, that   $\Sres_d(f,g)=0$, and for $d=\ell$,
$\Sres_\ell(f,g)=g=\Sres_m(f,g)$.

\smallskip
 We have  the universal subresultant B\'ezout identity
  \begin{equation}\label{bezout}\Sres_d(f,g)=F_d\,f+G_d\,g,\end{equation}
where
{\small
\begin{equation}\label{Fd}
F_d
:=\det%
\begin{array}{|cccccc|c}
\multicolumn{6}{c}{\scriptstyle{2\ell+1-2d}}\\
\cline{1-6}
f_{\ell+1} & \cdots & & \cdots & f_{d+1-\left(\ell-d-1\right)}& x^{\ell-d-1} &\\
& \ddots & && \vdots  & \vdots &\scriptstyle{\ell-d}\\
&  &f_{\ell+1}&\cdots &f_{d+1}& x^0 & \\
\cline{1-6}
g_{\ell} &\cdots & & \cdots & g_{d+1-(\ell-d)}&0&\\
&\ddots &&&\vdots & \vdots &\scriptstyle{\ell+1-d}\\
&& g_{\ell} &\cdots & g_{d+1} &0&\\
\cline{1-6} \multicolumn{2}{c}{}
\end{array}\ ,
\end{equation}
and
\begin{equation}\label{Gd}
G_d
:=\det%
\begin{array}{|cccccc|c}
\multicolumn{6}{c}{\scriptstyle{2\ell+1-2d}}\\
\cline{1-6}
f_{\ell+1} & \cdots & & \cdots & f_{d+1-\left(\ell-d-1\right)}& 0&\\
& \ddots & && \vdots  & \vdots &\scriptstyle{\ell-d}\\
&  &f_{\ell+1}&\cdots &f_{d+1}& 0& \\
\cline{1-6}
g_{\ell} &\cdots & & \cdots & g_{d+1-(\ell-d)}&x^{\ell-d} &\\
&\ddots &&&\vdots & \vdots &\scriptstyle{\ell+1-d}\\
&& g_{\ell} &\cdots & g_{d+1} & x^0 &\\
\cline{1-6} \multicolumn{2}{c}{}
\end{array}\ .
\end{equation}}
Observe that $\deg(G_d)\le \ell-d,$ if $G_d\ne 0$.

\smallskip
The  result below expresses  the existence and uniqueness of the
solution of the osculatory rational interpolation problem in terms
of the subresultant sequence of $f$ and $g$.

\begin{theorem}\label{t3}
With notation as above, let $0\leq d\le a$ be the maximal index such
that $\Sres_d(f,g)\ne 0$. Then $\deg(G_d)\le b$ and the osculatory
rational interpolation problem \eqref{uv} has a solution  if and
only if $G_d(x_i)\ne 0$ for $1\le i\le k$. In that case the solution
is given by
\begin{equation*}\label{idmain}\frac{A}{B}=\frac{\Sres_d(f,g)}{G_d},\end{equation*}
where moreover $\gcd\big(\Sres_d(f,g),G_d\big)=1$.
\end{theorem}

This result is strongly related to  Theorem 5.16 from  \cite{vzGJ03}, which expresses  the existence and uniqueness of the
solution of the osculatory rational interpolation problem in terms
of the Extended Euclidean Scheme for  $f$ and $g$. We recall its statement below, as well as  Lemma  5.15 and a consequence of Lemma 3.15(v) of the same reference.

\begin{theorem} \label{vzg} \cite[Theorem 5.16, Lemmas 5.15 and 3.15(v)]{vzGJ03}
\\
With notation as above, let $r_i=s_i f + t_i g$, $i\ge 0$, be the successive remainders in  the Extended Euclidean
Scheme for $f$ and $g$, and $s_i,t_i$ the corresponding B\'ezout coefficients.
\begin{enumerate}
\item
The osculatory rational interpolation problem \eqref{uv} has a solution  $A/B$ if and only if the
minimal row $r_j=s_j f + t_j g$  such
that $ d_j:=\deg(r_j)\le a$  satisfies
$\gcd(r_j,t_j)=1$.
If this is the case, $A/B=r_j/t_j$ is the  solution
(and in particular $\deg(t_j)\le b$).
\item Let $r=sf+tg\ne 0$ be such that $\deg(r_j)\le \deg(r)< \deg(r_{j-1})$ and $\deg (r)+\deg (t) <\ell+1=\deg(f)$. Then there exists $c\in K$ such that $r=cr_j, s=cs_j, t=ct_j$. Moreover,
$\gcd(s,t)=1$.
\end{enumerate}
\end{theorem}


\begin{proof}[Proof of Theorem \ref{t3}.] We consider the minimal $j$ in the Extended Euclidean Scheme such that $ d_j:=\deg(r_j)\le a$: by Theorem \ref{vzg}\,(1), there is a solution $A/B=r_j/t_j$ to our problem if and only if $\gcd(r_j,t_j)=1$.
Observe that
for  $d_{j-1}:=\deg(r_{j-1})$  we have $a< d_{j-1}$, i.e. $d_j\le a<d_{j-1}$.

\noindent Let $d\le a$ be the largest such that $\Sres_d(f,g)\ne 0$. One has $\Sres_d(f,g)=F_df+G_dg$ with
$\deg(\Sres_d(f,g))+\deg(G_d)\le \ell<\ell+1= \deg(f)$. Moreover, by
the Fundamental Theorem of Polynomial Remainder Sequences, \cite{Coll67,BT71}
or \cite[Th.7.4]{GCL96}, $\Sres_{d_j}(f,g)$ and
$\Sres_{d_{j-1}-1}(f,g)$ are (non-zero) constant multiples of $r_j$
(and $\Sres_{d'}(f,g)=0$ for $d_j<d'<d_{j-1}-1$). This implies that $d_j\le d< d_{j-1}$. Therefore, applying Theorem \ref{vzg}\,(2), there exists $c\in K^\times$ such that
$$\Sres_d(f,g)=cr_j, \ F_d=c\,s_j \ \mbox{and} \ G_d=c\,t_j$$ with $\gcd(F_d,G_d)=1$. This implies, by the definition of $f$,
$$\gcd\big(\Sres_d(f,g),G_d\big)=1\ \Longleftrightarrow \
G_d(x_i)\ne 0 \ \mbox{ for } 0\le i\le k.$$
This concludes the proof.
\end{proof}

\begin{remark}
In the statement  of Theorem \ref{t3}, one can replace the hypothesis
{\em ``let $0\leq d\le a$ be the maximal index such that
$\Sres_d(f,g)\neq0$''} by {\em ``let $a\leq d\le \ell$ be the minimal
index such that $\Sres_d(f,g)\neq0$''}. This is due to the Fundamental Theorem of Polynomial Remainder Sequences mentioned in the previous proof, since 
if $\Sres_a(f,g)=0$, then one has that for $\Sres_k(f,g)$ and $\Sres_j(f,g)$ coincide up to a non-zero constant, for the
the maximal $k<a$ such that $\Sres_k(f,g)\ne 0$ and the minimal
$j>a$ such that $\Sres_j(f,g)\neq0$. 
Accordingly, one can replace the corresponding hypothesis in  Theorems \ref{mainroots} and
\ref{determinantal} below.
\end{remark}

\smallskip
Theorem \ref{t3} has the advantage that it can be applied to produce
explicit formulae for the Cauchy and the osculatory rational interpolation problems in terms of
the input data, as we show in the next sections.

\section{The Cauchy interpolation problem formula}
We now  present the closed expression in terms of the data for the
Cauchy interpolation problem.
 For $U,V\subset K$, we set $R(U,V):=\prod_{u\in U, v\in V}(u-v).$

\begin{theorem}\label{mainroots}
Given $(a,b)$, $X:=\{x_0, \ldots, x_\ell\}$ and $y_0,\ldots, y_\ell$
as in Problem \eqref{uu}. Let $d$ be maximal such that $0\leq d\leq
a$ and
$$
A_0:= \sum_{X'\subset X,|X'|=d}R(x,X')\big(\prod_{x_j\notin X'}y_j\big)/R(
X\setminus X',X') \  \in K[x]
$$
is not identically zero. Set $$ B_0:=\sum_{X''\subset X,
|X''|=\ell-d} R(X'',x)\big(\prod_{x_j\in X''}y_j\big)/R(X'',
X\setminus X'')\ \in K[x].
$$ Then $\deg(B_0)\le b$ and  a solution $\dfrac{A}{B}$ for the
Cauchy interpolation problem \eqref{uu}   exists if and only if
$B_0(x_i)\neq 0$ for  $0\le i\le \ell$. In that case the solution is
given by $$\frac{A}{B}=\frac{A_0}{B_0}.$$
\end{theorem}

\begin{proof}
Let as before
$f=\prod_{i=0}^\ell(x-x_i),$ and $g$ be
the unique polynomial of degree bounded by $\ell$ which satisfies
$g(x_i)=y_i$ for $0\le i\le \ell$. Denote by $Z$ the set of roots of $g$ in $\overline{K},$ the algebraic closure of $K$.
\\
 Let $d$ be maximal such that $0\leq d\leq
a$ and $\Sres_d(f,g)\ne 0$. We apply Theorem \ref{t3} and 
Sylvester's single-sum formula   in roots for
$\Sres_d(f,g)$ (see for instance the original paper of Sylvester
\cite[Art. 21]{sylv} or the many other references on the topic) and
for
 $G_d$  (\cite[Art. 29]{sylv}, or \cite{KS12}, Remark after Lemma~6): \\
{\small $$\aligned
&\operatorname*{Sres}\nolimits_{d}(f,g)=\sum_{|X^{\prime}|=d}R(x,X^{\prime})\;\frac{R(X\setminus X^{\prime}%
,Z)}{R(X\setminus X^{\prime},X^{\prime})} \ = \
\sum_{|X^{\prime}|=d}R(x,X^{\prime})\;\frac{\prod_{x_j\notin X'}
g(x_j)}{R(X\setminus X^{\prime},X^{\prime})} \\
& \qquad \qquad \qquad
=\sum_{|X^{\prime}|=d}R(x,X^{\prime})\;\frac{\prod_{x_j\notin
X'} y_j}{R(X\setminus X^{\prime},X^{\prime})}\ = \ A_0,\\
&G_d=(-1)^{\ell-d}
  \sum_{
|X''|=\ell-d}R(x,X'')\,\frac{R(X'', Z)}{R(X'',X\setminus X'' )}\\ &
\qquad \quad = \ \sum_{ |X''|=\ell-d}R(X'', x)\,\frac{R(X'',
Z)}{R(X'',X\setminus X'')} \  = \    \sum_{
|X''|=\ell-d}R(X'',x)\frac{\prod_{x_j\in
X''}g(x_j)}{R(X'',X\setminus X'')}\\
& \qquad \quad =  \sum_{ |X''|=\ell-d}R(X'',x)\frac{\prod_{x_j\in
X''}y_j}{R(X'',X\setminus X'')}\ = \ B_0.
\endaligned
$$}
 where both  $X',\,X''\subset X$.
 The claim follows from  Theorem \ref{t3}.
\end{proof}

\begin{remark}
Observe that when $a=\ell$ then
$\Sres_\ell(f,g)=g\ne 0$ and  Theorem \ref{mainroots} specializes to
the well-known {\em Lagrange interpolation polynomial}  associated
to the data $\{(x_i, y_i)\}_{0\leq i\leq\ell}$, that is
$$\frac{A_0}{B_0}=\sum_{0\le
i\le \ell}y_i\frac{\prod_{j\ne i} (x-x_j)}{\prod_{j\ne i}
(x_i-x_j)}={\sum_{0\le i\le
\ell}y_i\frac{R(x,X\setminus\{x_i\})}{R(x_i,X\setminus\{x_i\})}}.$$
\end{remark}

\smallskip
The gap $d<a$ in Theorem \ref{t3} may  appear, as the following example shows.

\begin{example} We consider the Cauchy interpolation problem with
 $a=3, \, b=2$, and  the associated input data
\begin{align*} X&=\big( x_0,\dots,x_5\big)  \mbox{ where }
x_0,\dots, x_5   \mbox{ are the } 6 \mbox{ different roots of }  x^6-1  \mbox{ in } \overline K,\\
Y&=(y_0,\dots,y_5) \ \mbox{ with }\ y_i=x_i^5 + 2 \ \mbox{ for } \
0\le i\le 5.\end{align*} for a field $K$ of characteristic $\ne
2,3$. In this case we have
$$
f=x^6-1 \quad \mbox{and} \quad  g=x^5+2.
$$
An explicit computation shows that $\Sres_3(f,g)=\Sres_2(f,g)=0.$
However,
$$
\Sres_1(f,g)= 8+16x, \quad F_1=-8, \quad  G_1=8x.
$$
We easily verify that $G_1(x_i)\ne 0$ for $0\le i\le 5$.  Hence by
Theorem \ref{t3}, $d=1$ and
$$\frac{A}{B}=\frac{8+16x}{8x} =
\frac{1+2x}{x}
$$
is the  solution to this Cauchy interpolation problem, which can be
checked straightforwardly since
$$
\frac{1+2x_i}{x_i}=\frac{1}{x_i}+2=x_i^{5}+2=y_i,\  \ i=0,\ldots, 5.
$$
\end{example}
\section{The osculatory rational interpolation formula}
Before stating our main result for the osculatory rational interpolation problem,
we need to set a  notation.

\begin{notation}\label{not}
Set $a, b\in \N$ such that $a+b=\ell$, $a_0, \ldots, a_k\in \N$ such
that $a_0+\cdots +a_k=\ell+1$, as in Problem \eqref{uv}. We define
\begin{itemize}

\item $\ov {X}:=\big((x_0, a_0);\dots;(x_k,a_k)\big)$ an array of pairs in $ K\times \N$ and
$Y:=\big(\bfy_0,\dots,\bfy_k\big)$ where
$\bfy_i=(y_{i,0},\dots,y_{i,a_i-1})$. We call $(\overline X,Y)$  the
{\em input data} for the osculatory rational interpolation problem.
\item Set $u\in \N$. The {\em generalized Vandermonde}
or {\em confluent matrix} (e.g.  \cite{Kal1984}) of size $u+1$  associated to
$\overline X $ is the (non-necessarily square)
matrix $V_{u+1}(\overline X)\in K^{(u+1)\times(\ell+1)}$ defined by
{\small $$
V_{u+1}(\overline
X):=\begin{array}{|c|c|c|c}
\multicolumn{3}{c}{\scriptstyle{\ell+1}}\\
\cline{1-3} & & & \\ V_{u+1}(x_0,a_0) & \dots &V_{u+1}(x_k,a_k)
 &\scriptstyle u+1\\ & & & \\
 \cline{1-3}
 \multicolumn{2}{c}{}
\end{array},
$$}
where for any $t$, $V_{u+1}(x_i, t+1)\in K^{(u+1)\times (t+1)}$ is defined by
{\small $$
V_{u+1}(x_i, t+1):=\begin{array}{|ccccc|c}
\multicolumn{5}{c}{\scriptstyle{t+1}}\\
\cline{1-5}
1&0 &0&\dots & 0 &\\
x_i&1 &0&\dots & 0 &\\
x^2_i&2x_i &1&\dots & 0 &\scriptstyle u+1\\
\vdots&  \vdots& \vdots & &\vdots &\\
x_i^u&ux_i^{u-1} &{u\choose 2}x_i^{u-2}&\dots & {u\choose t}x_i^{u-t}& \\
 \cline{1-5}
 \multicolumn{2}{c}{}
\end{array}
$$}

\item We define the matrix $U_{u+1} (\overline X, Y)\in K^{(u+1)\times (\ell+1)}$ associated to $\ov {X}$  and
$Y$ as: {\small$$ U_{u+1}(\overline X, Y):=\begin{array}{|c|c|c|c}
\multicolumn{1}{c}{}&\multicolumn{1}{c}{\scriptstyle{\ell+1}}& \multicolumn{1}{c}{}& \\
\cline{1-3} & & & \\ U_{u+1}(x_0; \bfy_0) & \dots &U_{u+1}(x_k;
\bfy_k)
 &\scriptstyle u+1\\ & & & \\
 \cline{1-3}
 \multicolumn{2}{c}{}
\end{array},
$$}
where for any $t$, $U_{u+1}(x_i,\bfy_i)\in K^{(u+1)\times{(t+1)}}$
is defined by {\small $$U_{u+1}(x_i,\bfy_i)=
\begin{array}{|cccc|c}
\multicolumn{4}{c}{\scriptstyle{t+1}}\\
\cline{1-4}
y_{i,0}&y_{i,1} &\dots & y_{i,t} &\\
y_{i,0}x_i&y_{i,1}x_i+y_{i,0} &\dots & y_{i,t}x_i  +  y_{i,t-1}& \\
\vdots&  \vdots& &\vdots& \scriptstyle u+1 \\
y_{i,0}x_i^{u}&y_{i,1}x_i^{u}+uy_{i,0}x_i^{u-1} &\dots &
\sum_{j=0}^{t}{u\choose j}  y_{i,t-j} x_i^{u-j}& \\
 \cline{1-4}
 \multicolumn{2}{c}{}
\end{array}\ ,
$$}
\end{itemize}
where
$$ \big(U_{u+1}(x_i,\bfy_i)\big)_{k+1,\,l+1}= \sum_{j=0}^l{k\choose
j}y_{i,\,l-j}x_i^{k-j},$$
with the convention that when $u<j$, ${u\choose j}x_i^{u-j}=0$.
\end{notation}

The next determinantal expression presents  the solution of the
osculatory rational interpolation problem in terms of the input data
as follows:

\begin{theorem}\label{determinantal}  Under the notation above, let $d$ be maximal such that $0\leq d \leq a$ and
\begin{equation}\label{A}
A_0:=- \det \
\begin{array}{|c|c|c}
\multicolumn{1}{c}{ \scriptstyle{\ell+1}}&\multicolumn{1}{c}{ \scriptstyle{1}}\\
  \cline{1-2}&1&\scriptstyle{}\\
  V_{d+1}(\ov X)& \vdots &\scriptstyle{d+1}\\
  &x^d&\scriptstyle{}\\
\cline{1-2}\  &&\\U_{ \ell-d+1}(\ov X, Y)
 &\mathbf{0} &\scriptstyle{\ell-d+1}\\
&& \\
\cline{1-2} \multicolumn{2}{c}{}
\end{array} \in K[x]
\end{equation}
is not identically zero. Set
\begin{equation}\label{B}
B_0:=\det \
\begin{array}{|c|c|c}
\multicolumn{1}{c}{ \scriptstyle{\ell+1}}&\multicolumn{1}{c}{ \scriptstyle{1}}\\
\cline{1-2}\  &&\\V_{d+1}(\ov X)
 &\mathbf{0} &\scriptstyle{d+1}\\
&& \\
  \cline{1-2}&1&\scriptstyle{}\\
  U_{\ell-d+1}(\ov X, Y)& \vdots &\scriptstyle{\ell-d+1}\\
  &x^{\ell-d}&\scriptstyle{}\\
  \cline{1-2} \multicolumn{2}{c}{}
\end{array}\in K[x],
\end{equation}
Then $\deg(B_0)\le b$, and a solution $\dfrac{A}{B}$ for the
osculatory rational interpolation problem \eqref{uv} exists if and
only if  $B_0(x_i)\neq 0$ for  $0\le i\le  k$. In that case the
solution is given by $$\frac{A}{B}=\frac{A_0}{B_0}.$$
\end{theorem}

To prove this result  we need  the following lemma, that we prove at the end of the section.
\begin{lemma}\label{lem} Let $d\le \ell$. Then
{\small $$ \Sres_d(f,g)=(-1)^{\ell+1-d}\,\det\big(V_{\ell+1}(\ov
X)\big)^{-1}\,
 \det \
\begin{array}{|c|c|c}
\multicolumn{1}{c}{ \scriptstyle{\ell+1}}&\multicolumn{1}{c}{ \scriptstyle{1}}\\
  \cline{1-2}&1&\scriptstyle{}\\
  V_{d+1}(\ov X)& \vdots &\scriptstyle{d+1}\\
  &x^d&\scriptstyle{}\\
\cline{1-2}\  &&\\U_{ \ell-d+1}(\ov X,Y)
 &\mathbf{0} &\scriptstyle{\ell-d+1}\\
&& \\
\cline{1-2} \multicolumn{2}{c}{}
\end{array}.
$$}
and {\small $$ G_d=(-1)^{\ell - d} \det\big(V_{\ell+1}(\ov
X)\big)^{-1}
\det \
\begin{array}{|c|c|c}
\multicolumn{1}{c}{ \scriptstyle{\ell+1}}&\multicolumn{1}{c}{ \scriptstyle{1}}\\
\cline{1-2}\  &&\\V_{d+1}(\ov X)
 &\mathbf{0} &\scriptstyle{d+1}\\
&& \\
  \cline{1-2}&1&\scriptstyle{}\\
  U_{\ell-d+1}(\ov X,Y)& \vdots &\scriptstyle{\ell+1-d}\\
  &x^{\ell-d}&\scriptstyle{}\\
  \cline{1-2} \multicolumn{2}{c}{}
\end{array}.
$$}
\end{lemma}

\begin{proof}[Proof of Theorem \ref{determinantal}]
The maximal index $d\le a$ such that $A_0\ne 0$ clearly coincides
with the maximal index $d\le a$ such that $\Sres_d(f,g)\ne 0$, since
these two quantities only differ by a non-zero constant.
Analogously, $G_d(x_i)\ne 0 \Leftrightarrow B_0(x_i)\ne 0$. Finally,
$\dfrac{A_0}{B_0}=\dfrac{\Sres_d(f,g)}{G_d}$.
 \end{proof}

 \begin{remark}\label{LB} As we checked after a useful discussion
with George Labahn at the 2013 SIAM Conference on Applied Algebraic Geometry, the
matrix formulations for $A_0$ and $B_0$ in Theorem \ref{determinantal} 
can actually be
derived  from the  {\em Mahler systems} introduced by Beckermann and
Labahn in \cite{BL00} to solve a more general class of problems.
Indeed, by translating our situation into their general framework
(see \cite[Example 2.3]{BL00}), and using the standard Hermite dual
basis to produce their matrices, it can be seen that the
determinants appearing in the right hand side of \eqref{A} and
\eqref{B} coincide with those defining
$p^{(\ell)}(\stackrel{\rightarrow}{n},z)$ in \cite[Section 5]{BL00}.
\end{remark}

 \begin{remark} Let us note that in particular,  Problem \eqref{uv} for  $a=\ell$
corresponds to the ordinary {\em Hermite interpolation problem},
i.e. the determination of the {Hermite interpolation polynomial} $g$
associated to the input data $$\overline X=\big((x_0,a_0),\ldots,
(x_k, a_k)\big), Y=(\bfy_0,\dots,\bfy_k)\ \mbox{ where} \
\bfy_i=(y_{i,j})_{0\le j<a_i},$$ which is the  unique polynomial of
degree less than or equal to $\ell$ such that
 $$g^{(j)}(x_i)=j!\, y_{i,j},\, 0\le i\le  k,\ 0\le j<a_i.$$
 In this case,
 Theorem \ref{determinantal} specializes to the well-known
determinantal expression for the polynomial $g$, that is
\begin{equation*}\label{hhermite}
g=-\, \frac{ \det \
\begin{array}{|c|c|}
  \cline{1-2}
  & 1\\
  V_{\ell+1}(\ov X)& \vdots \\
  &x^\ell\\
\cline{1-2}\  Y
 &{0} \\
\cline{1-2}
\end{array} {\ } }{\det\big(V_{\ell+1}(\ov X)\big) {\ } }
,
\end{equation*}
setting in this way a unified determinantal framework for polynomial
and rational interpolation problems.
\end{remark}

\begin{remark}
We  remark that for the Cauchy interpolation problem
\eqref{uu}, the solution described by Theorem~\ref{determinantal}
gives { \small {{$$\frac{A}{B}\,=\,-\
\frac{\det\begin{array}{|ccc|c|l}\multicolumn{1}{c}{}&
\multicolumn{1}{c}{\scriptstyle{\ell+1}}&\multicolumn{1}{c}{}&\multicolumn{1}{c}{
\scriptstyle{1}}&
\\\cline{1-4}
1&\cdots&1&1&\\
\vdots& &\vdots&\vdots&{\scriptstyle{d+1}}\\
x_0^d&\cdots&x_\ell^d&x^{d}&\\
\cline{1-4}  y_{0}&\ldots&y_{\ell}&0&{\scriptstyle{ }}\\
\vdots& &\vdots&\vdots&{\scriptstyle{\ell-d+1}}\\
 y_0 x_0^{\ell-d}&\dots &y_\ell x_\ell^{\ell-d} & 0                   \\ \cline{1-4}  \multicolumn{2}{c}{}
                     \end{array}}{\det\begin{array}{|ccc|c|l}\multicolumn{1}{c}{}&
\multicolumn{1}{c}{\scriptstyle{\ell+1}}&\multicolumn{1}{c}{}&\multicolumn{1}{c}{
\scriptstyle{1}}&
\\\cline{1-4}
1&\cdots&1&0&\\
\vdots& &\vdots&\vdots&{\scriptstyle{d+1}}\\
x_0^d&\cdots&x_\ell^d&0&\\
\cline{1-4}  y_{0}&\ldots&y_{\ell}&1&{\scriptstyle{ }}\\
\vdots& &\vdots&\vdots&{\scriptstyle{\ell-d+1}}\\
 y_0 x_0^{\ell-d}&\dots &y_\ell x_\ell^{\ell-d}& x^{\ell-d}                  \\ \cline{1-4}  \multicolumn{2}{c}{}
                     \end{array}} \ .$$}}}

\end{remark}

\smallskip
\begin{example}
We consider the  osculatory rational interpolation problem  with
$a=b=2$, $k=2$, and the associated input data
\begin{align*}\overline X&=\big( (x_0,a_0),(x_1,a_1)\big) \ \mbox{ with }\
(x_0,a_0)=(1,2), \, (x_1,a_1)= (2,3)\\
Y&=(\bfy_0,\bfy_1) \ \mbox{ with }\
\bfy_0=(y_{0,0},y_{0,1})=(2,3),\,
\bfy_1=(y_{1,0},y_{1,1},y_{1,2})=(6,7,8).\end{align*} We have
$$
f=(x-1)^2(x-2)^3 \quad \mbox{and} \quad  g= -8 + 23 x - 20 x^2 + 8
x^3 - x^4.$$ By explicit computation, we get
$$
\Sres_2(f,g)=35 x - 25 x^2, \quad F_2=-25 + 5 x ,\quad  G_2=25 - 25
x + 5 x^2.
$$
We easily verify that  $G_2(x_i)\ne 0$ for $i=1,2$ if $5\ne 0$ in
$K$. Hence by Theorem \ref{t3}, $d=a=2,$ and
$$\frac{A}{B}=\frac{35x -  25x^2}{25 -  25x + 5x^2} =
\frac{7x-5x^2}{5-5x+x^2}
$$
is the  solution to the rational interpolation problem, which can be
checked straightforwardly.
\end{example}

\smallskip
\begin{proof} [Proof of Lemma \ref{lem}]  By   \cite[Theorem 2.5]{DKS13}, 
{\small $$ \Sres_d(f,g)=(-1)^{\ell+1-d}\,\det\big(V_{\ell+1}(\ov
X)\big)^{-1}\,
 \det \
\begin{array}{|c|c|c}
\multicolumn{1}{c}{ \scriptstyle{\ell+1}}&\multicolumn{1}{c}{ \scriptstyle{1}}\\
  \cline{1-2}&1&\scriptstyle{}\\
  V_{d+1}(\ov X)& \vdots &\scriptstyle{d+1}\\
  &x^d&\scriptstyle{}\\
\cline{1-2}\  &&\\W_{g, \ell-d+1}(\ov X)
 &\mathbf{0} &\scriptstyle{\ell-d+1}\\
&& \\
\cline{1-2} \multicolumn{2}{c}{}
\end{array},
$$}
where $W_{g,\ell-d+1}(\ov X)$ is   the {\em generalized Wronskian}
of size $\ell-d+1$ associated to $\overline X $, i.e. the matrix
{\small
$$
W_{g,\ell-d+1}(\overline X):=\begin{array}{|c|c|c|c}
\multicolumn{1}{c}{}&\multicolumn{1}{c}{\scriptstyle{\ell+1}}& \multicolumn{1}{c}{}& \\
\cline{1-3} & & & \\ W_{g,\ell-d+1}(x_0,a_0) & \dots
&W_{g,\ell-d+1}(x_k, a_k)
 &\scriptstyle \ell-d+1\\ & & & \\
 \cline{1-3}
 \multicolumn{2}{c}{}
\end{array} \ \in \ K^{(\ell-d+1)\times (\ell+1)}
$$}
where  {\small $$ W_{g, \ell-d+1}(x_i,a_i):=\begin{array}{|cccc|c}
\multicolumn{4}{c}{\scriptstyle{a_i}}\\
\cline{1-4}
g(x_i)&g'(x_i) &\dots & \frac{g^{(a_i-1)}(x_i)}{(a_i-1)!} &\\
(xg)(x_i)&(zg)'(x_i) &\dots & \frac{(xg)^{(a_i-1)}(x_i)}{(a_i-1)!} & \scriptstyle \ell-d+1\\
\vdots&  \vdots& &\vdots& \\
(x^{\ell-d}g)(x_i)&(x^{\ell-d}g)'(x_i) &\dots & \frac{(x^{\ell-d}g)^{(a_i-1)}(x_i)}{(a_i-1)!} & \\
 \cline{1-4}
 \multicolumn{2}{c}{}
\end{array}\  \in \ K^{(\ell-d+1)\times a_i}
$$}
(with the convention that when $k<j$, ${k\choose j}x_i^{k-j}=0$).
\\
In the same way, we prove now that {\small $$ G_d=(-1)^{\ell - d}
\det\big(V_{\ell+1}(\ov X)\big)^{-1}
\det \
\begin{array}{|c|c|c}
\multicolumn{1}{c}{ \scriptstyle{\ell+1}}&\multicolumn{1}{c}{ \scriptstyle{1}}\\
\cline{1-2}\  &&\\V_{d+1}(\ov X)
 &\mathbf{0} &\scriptstyle{d+1}\\
&& \\
  \cline{1-2}&1&\scriptstyle{}\\
  W_{g,\ell-d+1}(\ov X)& \vdots &\scriptstyle{\ell+1-d}\\
  &x^{\ell-d}&\scriptstyle{}\\
  \cline{1-2} \multicolumn{2}{c}{}
\end{array}.
$$}
For that, we  consider the following matrices:
 {\small
\[\begin{array}{cc}
M_{ f}:=
\begin{array}{|ccccc|c|c}
\multicolumn{6}{c}{\scriptstyle{2\ell-d+2}}\\
\cline{1-6}
f_{0} & \dots & f_{\ell+1} &  & &0&\\
& \ddots &  & \ddots & &\vdots&\scriptstyle{\ell-d}\\
&  & f_{0} & \dots & f_{\ell+1} &0&\\
\cline{1-6} \multicolumn{2}{c}{}
\end{array},
\\  M_{g}:=
\begin{array}{|ccccc|c|c}
\multicolumn{6}{c}{\scriptstyle{2\ell-d+2}}\\
\cline{1-6}
g_{0} & \dots & g_{\ell} &  & &x^{0}&\\
& \ddots &  & \ddots & &\vdots &\scriptstyle{\ell+1-d}\\
&  & g_{0} & \dots & g_{\ell} &x^{\ell-d}& \\
\cline{1-6} \multicolumn{2}{c}{}
\end{array}
\end{array}
\]}
and {\small $$U_{d}:=%
\begin{array}{|c|c}
\multicolumn{1}{c}{ \scriptstyle{2\ell-d+2}}&\\
\cline{1-1}
\ I_{d+1}\ \ &\scriptstyle{d+1}\\
\cline{1-1}
M_{ f}&\scriptstyle{\ell-d}\\
\cline{1-1}
M_{ g}&\scriptstyle{\ell+1-d}\\
\cline{1-1} \multicolumn{2}{c}{}
\end{array} ,$$}
where $I_{d+1}$ is the $(d+1)\times(2\ell-d+2)$ matrix with the
identity matrix on the left and zero otherwise. Then from the
definition of $G_d$ we have that
$$
G_d= \det (U_d).
$$
Also, similarly as in the proof of  \cite[Theorem 2.5]{DKS13}, we have {\small $$
\begin{array}{c|c|}
\multicolumn{1}{c}{}&
\multicolumn{1}{c}{\scriptstyle{2\ell-d+2}}\\
\cline{2-2}
\scriptstyle{d+1}& \ \ I_{d+1}\ \ \\
\cline{2-2}
\scriptstyle{\ell-d}& M_{ f}\\
\cline{2-2} \scriptstyle{\ell+1-d}&
M_{ g}\\
\cline{2-2} \multicolumn{2}{c}{}
\end{array}
\,
\begin{array}{|c|c|c}
\multicolumn{1}{c}{ \scriptstyle{\ell+1}}&\multicolumn{1}{c}{ \scriptstyle{\ell-d+1}}&\\
\cline{1-2} & \mathbf{0} &\scriptstyle{\ell+1}\\
V_{2\ell-d+1}(\ov X )&\\
\cline{2-2}&&\\
\cline{1-1}
0&\Id_{\ell-d+1} &\scriptstyle{\ell-d+1}\\
\cline{1-2} \multicolumn{3}{c}{}
\end{array}\   =    \  \begin{array}{|c|c|c|c}
\multicolumn{1}{c}{ \scriptstyle{\ell+1}}&\multicolumn{1}{c}
{ \scriptstyle{\ell-d}}&\multicolumn{1}{c}{ \scriptstyle{1}}&\\
\cline{1-3}&&&\\
  V_{d+1}(\ov X)& * &0&\scriptstyle{d+1}\\
  &&&\\
\cline{1-3}
&&&\\
 \mathbf{0}&M'_{ f}&0&\scriptstyle{\ell-d}\\
 &&&\\
\cline{1-3}
&&1&\\
\ W_{ g,\ell+1-d}(\ov X)&* &\vdots&\scriptstyle{\ell+1-d}\\
&&x^{\ell-d}\\
\cline{1-3} \multicolumn{3}{c}{}
\end{array},
$$}
where $M'_f$ is a triangular matrix with $f_{\ell+1}=1$ in its
diagonal. This shows the formula for $G_d$.
\\

\noindent Finally, we simply show that $W_{g,\ell-d+1}(\overline
X)=U_{\ell-d+1}(\overline X,Y)$ by computing the entries of
$W_{g,\ell-d+1}(\overline X)$: we apply    Leibniz rule and the fact
that $g^{(t-j)}(x_i)=(t-j)!\, y_{i,j}$ for $0\le i\le k$ and $0\le
j<a_i$:
$$
 \frac{(x^{u}g)^{(t)}(x_i)}{t!}= \sum_{j=0}^{t}{u \choose j}  x_i^{u-j}y_{i,t-j}.
 $$
\end{proof}

\smallskip

\begin{ac}
We are grateful to Alain Lascoux for having explained us part of the
results in \cite{las}. In December 2012, a preliminary version of
this paper (\cite{DKS12}) was posted in the arxiv, and its results
were further communicated in both the MEGA 2013 conference and the
2013 SIAM Algebraic Geometry Meeting, where we received a lot of
comments and suggestions for improvements. In particular, we are
grateful to George Labahn for having discussed with us his previous
results on the topic, \cite{BL00}, as commented in Remark \ref{LB}. We
also thank   Bernard Mourrain for very helpful suggestions for
future projects, and the referees for helping us improving the presentation of the results. All the examples and computations have been worked
out with the aid of the software {\tt Mathematica 8.0}
(\cite{math}).
\end{ac}

\bigskip
\bibliographystyle{alpha}
\def\cprime{$'$} \def\cprime{$'$} \def\cprime{$'$}

\end{document}